\documentclass[12pt]{article}
\usepackage[centertags,intlimits]{amsmath}                     
\usepackage{color}\usepackage{amsfonts}                     
\usepackage{amsthm}
\usepackage{graphics}
\allowdisplaybreaks[2]
\numberwithin{equation}{section}
\newtheorem{theorem}{Theorem}[section]
\newtheorem{lemma}{Lemma}[section]

\theoremstyle{remark}
\newtheorem{remark}{Remark}[section]

\providecommand{\abs}[1]{\lvert #1\rvert}

\newcommand{\nc}{\newcommand}
\nc{\vb}{\mathbf{v}}
\nc{\bx}{\mathbf{x}}
\nc{\by}{\mathbf{y}}
\nc{\bz}{\mathbf{z}}
\nc{\bu}{\mathbf{u}}
\nc{\bv}{\mathbf{v}}
\nc{\ba}{\mathbf{a}}
\nc{\bs}{\mathbf{s}}
\nc{\bq}{\mathbf{q}}
\nc{\bd}{\mathbf{d}}
\nc{\bb}{\mathbf{b}}
\nc{\bc}{\mathbf{c}}
\nc{\bi}{\mathbf{i}}
\nc{\bfr}{\mathbf{r}}
\nc{\bA}{\mathbf{A}}
\nc{\R}{\mathbb R}
\nc{\N}{\mathbb N}
\nc{\C}{\mathbb C}
\nc{\D}{\mathbb D}
\nc{\Z}{\mathbb Z}
\nc{\F}{\mathbf F}
\nc{\bbS}{\mathbb S}
\nc{\B}{\cal B}
\nc{\br}{\bigr}
\nc{\bl}{\bigl}
\nc{\Bl}{\Bigl}
\nc{\Br}{\Bigr}
\nc{\ind}{\mathbf{1}}
\nc{\bP}{\mathbf{P}}

\title{Large deviations of the long term  distribution of\\ a non Markov
  process}
\author{Anatolii A. Puhalskii \footnote{Email: puhalski@iitp.ru}\\
 Institute for Problems in Information
Transmission}
\begin{document}

\maketitle
\sloppy
\begin{abstract}
We prove that the long term distribution of the queue length
 process
in an ergodic generalised Jackson network obeys the Large Deviation
Principle with a  deviation function given by the
quasipotential. The latter  is related  to the unique long term
 idempotent
distribution, which is also a stationary idempotent distribution,
 of the large deviation limit of the queue length process. The
proof draws on developments in queueing network stability and
idempotent probability. 
\end{abstract}

\section{Introduction and summary}
\label{sec:introduction}

In a seminal contribution, Freidlin and Wentzell  \cite{wf} obtained the Large
Deviation Principle (LDP) for the stationary distribution of a 
diffusion process and showed that the deviation function,
which is often referred to as  the
action functional or the (tight) rate function, is given
by the quasipotential. Their ingenious analysis relied heavily
on the strong Markov property and involved an intricate study of attainment
times. Shwartz and Weiss \cite{SchWei95} adapted the methods of
Freidlin and Wentzell \cite{wf} to the
setting of  jump
Markov processes.
 In Puhalskii \cite{Puh03}, we suggested a
different, arguably, more direct and, as we hope, more robust
 approach.
 It was prompted by the analogy between large
deviations and weak convergence and sought to identify the 
deviation function  in terms of
the stationary idempotent distribution of a large deviation limit.
In this paper, the approach is applied to establishing the LDP for the
long term distribution of the non Markov process of
 queue lengths in a generalised Jackson network. 
It is noteworthy that, in addition to being non Markovian,
 generalised Jackson networks  fall into the
category of stochastic  systems with discontinuous dynamics, whose
analysis is generally more difficult.
We  show that the deviation function is still given by the
quasipotential which is related to the stationary idempotent
distribution of the limit idempotent process. That stationary
idempotent distribution  is also a unique long term idempotent
distribution, the uniqueness being proved by a coupling argument.
Geometric ergodicity of the queue length process enables us to
conclude that the long term idempotent distribution is the large deviation limit of
the long term queue length distributions.

\section{The setup and main result}
\label{sec:setup-main-results}

We consider a  queueing network with
 a homogeneous customer population which comprises $K$   
 single server stations.   
Customers arrive exogenously at the stations and are served there in the
order of arrival, one customer at a time. Upon being served, they either 
join a queue at another station or leave the network.
Let $A_k(t)$  denote the   
cumulative number of exogenous arrivals  at station $k$ 
by time $t$\,,
let $S_k(t)$ denote the cumulative number of customers   
that are served   
at station $k$ for the first $t$ units of busy time of that
station, and let 
$R_{kl}(m)$
denote the cumulative number of customers among the   
first $m$ customers  departing station $k$ that go directly to station   
$l$. Let
$A_k=(A_k(t),\,t\in\R_+)$,
$S_k=(S_k(t),\,t\in\R_+)$, and $R_k=(R_{k}(m),\,m\in\Z_+)$,
where $R_{k}(m)=(R_{kl}(m),\,l\in\mathcal{K})$ and
 $\mathcal{K}=\{1,2,\ldots, K\}$\,. 
It is assumed that the $A_k$ and $S_k$ are nonzero renewal processes
and
$  R_{kl}(m)=\sum_{i=1}^m 1_{\{\zeta_{k}^{(i)}=l\}}$,
where $\{\zeta_k^{(1)},\zeta_k^{(2)},\ldots\}$ is  a sequence of
i.i.d. random variables assuming values in  $\mathcal{K}$\,,
$1_\Gamma$ standing for the indicator function of set $\Gamma$\,.
The random entities $A_k$\,, $S_l$\,, $R_i$ and $Q(0)$ are
assumed
to be defined on  common probability
space $(\Omega,\mathcal{F},\mathbf{P})$ and
be mutually independent,  where 
$k,l,i\in\mathcal{K}$\,.
We  denote $p_{kl}=\mathbf P(\zeta_k^{(1)}=l)$ and let $P=(p_{kl})_{k,l=1}^K$\,.
The matrix $P$ is assumed to be of spectral radius less than unity so
that every arriving customer eventually leaves.
 Let $Q=(Q(t),\,t\in\R_+)$ represent 
  the queue length process,
where   $Q(t)=(Q_k(t),\,k\in\mathcal{K})$ and 
$Q_{k}(t)$ represents the number of customers  at station $k$ at time   
$t$\,.  
All  the stochastic processes are assumed  to have
piecewise constant right--continuous with
left--hand limits trajectories. Accordingly, they are considered as random
elements of the associated Skorohod spaces.

For $k\in\mathcal{K}$   and $t\in\R_+$, 
the following equations are satisfied:
\begin{equation}
  \label{eq:2}
Q_k(t)=Q_k(0)+A_k(t)+\sum_{l=1}^KR_{lk}\bl(D_{l}(t)\br)-D_k(t),
\end{equation}
where 
\begin{equation}
  \label{eq:1}
D_k(t)=S_k\bl(B_k(t)\br)
\end{equation}
represents  the number of  departures  from station $k$ by time $t$ and
\begin{equation}
  \label{eq:5}
  B_k(t)=\int_0^t1_{\{Q_k(u)>0\}}\,du
\end{equation}
represents the cumulative busy time of station $k$ by time $t$\,.
For  given realisations of
 $Q_k(0)$, $A_k$,  $S_k$, and $R_k$, 
there exist unique  $Q_k=(Q_k(t),\,t\in\R_+)$, 
$D_k=(D_k(t),\,t\in\R_+)$ and $B_k=(B_k(t),\,t\in\R_+)$ 
that satisfy 
 \eqref{eq:2}, \eqref{eq:1} and \eqref{eq:5}, see, e.g.,
Chen and Mandelbaum \cite{MR1106809}. 
The process $Q$ is non Markov unless all  $A_k$ and $S_k$ are Poisson
processes.

Let, for 
$k\in\mathcal{K}$\,, nonnegative random variables $\xi_k$ and
$\eta_k$ represent generic times between exogenous arrivals  and service times at station $k$, respectively. 
We assume that  
$\mathbf{P}(\xi_k=0)=\mathbf{P}(\eta_k=0)=0,$
$\mathbf{E}\exp(\theta \xi_k)<\infty$ and 
$\mathbf{E}\exp(\theta \eta_k)<\infty$  for some 
$\theta>0$, and the cumulative distribution functions of the 
$\xi_k$ and $\eta_k$ are right--differentiable at $0$ with 
 positive derivatives. 
Let
$\beta_k=\sup\{\theta\in\R_+:\,\mathbf{E}\exp(\theta
\xi_k)<\infty\}$ and
$\gamma_k=\sup\{\theta\in\R_+:\,\mathbf{E}\exp(\theta
\eta_k)<\infty\}$\,. 
Let also $\pi(u)=u\ln u-u+1$ if $u>0$\,, $\pi(0)=1$\,,
$\pi(\infty)=\infty$\,,  $0/0=0$\,, and $\infty\cdot 0=0$\,. 
Let $\mathbb{S}_+^{K\times K}$ represent the set of
row--substochastic $K\times K$--matrices and
 $I$ represent the $K\times K$--identity matrix.
Given vectors $\alpha=(\alpha_1,\ldots,\alpha_K)^T\in\R_+^K$
and 
$\delta=(\delta_1,\ldots,\delta_K)^T\in\R_+^K$,  matrix $\varrho\in\mathbb{S}_+^{K\times K}$ with
rows $\varrho_k,\,k\in\mathcal{K}$, and $J\subset\mathcal{K}$,
we define 
    \begin{align*}
 \psi^A_k(\alpha_k)&=\sup_{\theta<\beta_k}\bl(\theta- \alpha_k\ln \mathbf{E}
\exp(\theta \xi_k)\br)\,,\\
\psi^S_k(\delta_k)&=\sup_{\theta<\gamma_k}\bl(\theta- \delta_k\ln \mathbf{E}
\exp(\theta \eta_k)\br)\,,\\
\psi^R_k(\varrho_{k})&=
\sum_{l=1}^K \pi\Bl(\frac{\varrho_{kl}}{p_{kl}}\Br)p_{kl}+
\pi\Bl( \frac{1-\sum_{l=1}^K
  \varrho_{kl}}{1-\sum_{l=1}^K p_{kl} }\Br)\bl(1-\sum_{l=1}^K p_{kl}\br)
\end{align*}
and
\begin{equation}
\label{eq:psi_J}
     \psi_J(\alpha,\delta,\varrho)= \sum_{k=1}^K\psi_k^A(\alpha_k)+
\sum_{k\in J^c}\psi^S_k(\delta_k)
+\sum_{k\in J}\psi^S_k(\delta_k)
\,1_{\{\delta_k>\mu_k\}}
+\sum_{k=1}^K \delta_k\, \psi^R_k(\varrho_k)\,,
\end{equation}
where $J^c=\mathcal{K}\setminus J\,$.
Also, for   $y\in\R^K$, we let
\begin{equation}  \label{eq:Psi}
  \Psi_J(y)=\inf_{\substack{(\alpha,\delta,\varrho)\in\R_+^K\times\R_+^K\times
\mathbb{S}_+^{K\times K} 
:\\y=\alpha+(\varrho^T-I)\delta}}\psi_J(\alpha,\delta,\varrho)\,.
\end{equation}
If  $\,J\,$ is  a nonempty subset of $\mathcal{K}$,  
we denote 
$F_J=\{x=(x_1,\ldots,x_K)\in \R_+^K:  
x_k=0,k\in J,x_k>0,k\not\in J\}\,$,
$F_\emptyset$ is defined to be the interior of $\R_+^K$\,.
 Let, for
$x\in\R_+^K$ and $y\in\R^K$\,,
\begin{equation}
  \label{eq:3}
  L(x,y)=\sum_{J\subset\mathcal{K}}
1_{ F_J}(x)\Psi_J(y)\,.
\end{equation}
The function $L(x,y)$ is seen to be nonnegative.

Let \begin{equation*}
  \mathbf{I}_{x}(q)=\int_0^\infty
L(q(t),\dot q(t))\,dt\,,
\end{equation*}
provided $q=(q(t)\,,t\in\R_+)\in\D(\R_+,\R_+^K)$ is  
absolutely continuous  with
$q(0)=x\in\R_+^K$ and $\mathbf{I}_{x}(q)=\infty$\,, otherwise,
 where
 $q(t)=(q_1(t),\ldots,q_K(t))^T$\,.

With  large deviations in mind, we will
 assume in the next theorem that the initial queue length depends on large parameter
$n$\,, so, superscript ''$n$'' will be used to denote the associated random
quantities, e.g., $Q^n(t)$ is the queue length vector at time $t$\,.
Theorem 2.2 in Puhalskii
\cite{Puh07} proves the following result. 
\begin{theorem}
  \label{the:gen_jack}
If, in addition, $\mathbf P(\abs{Q^n(0)/n-
  x}>\epsilon)^{1/n}\to0$ as $n\to\infty$\,, for all
$\epsilon>0$\,,
then    the  queue length processes 
$\{(Q^n(nt)/n\,,t\in\R_+)\,,n\in\N\}$ obey  the LDP 
  in $\mathbb D(\R_+,\R^K_+)$ for rate $n$
 with the deviation function $\mathbf{I}_{x}(q)$\,.
\end{theorem}
For $x\in\R_+^K$\,, we define the quasipotential by
\begin{equation}
  \label{eq:4}
\mathbf  V(x)=\inf_{t\in\R_+}\inf_{\substack{q\in \D(\R_+,\R_+^K):\\
\,q(t)=x}}\mathbf I_0(q)\,.
\end{equation}
In order to address an LDP for the stationary queue length
distribution, we assume that the network  is subcritical:
\begin{equation}
  \label{eq:13}
\mu>(I-P^T)^{-1}\lambda
\end{equation} where 
$\mu=(\mu_1,\ldots,\mu_K)^T$\,,
$\lambda=(\lambda_1,\ldots,\lambda_K)^T$\,,
$\mu_k=1/\mathbf E\eta_k$ and $\lambda_k=1/\mathbf E\xi_k$\,.
(Inequalities between  vectors or matrices 
are understood to hold entrywise.)
 In addition, 
 we assume that 
 \begin{enumerate}
 \item there exists number $\overline\eta>0$ such that
$\mathbf E(\eta_k-u|\eta_k>u)\le \overline\eta$\,, for
$k\in\mathcal{K}$ and $u>0$\,,
\item $\mathbf P(\xi_k>u)>0$\,, for 
$k\in\mathcal{K}$ and $u>0$\,,
\item for  $k\in\mathcal{K}$\,,
there exist nonnegative function $f_k(u)$ on $\R_+$ with
  $\int_0^\infty f_k(u)\,du>0$ and $m_k\in\N$ 
such that $\mathbf P(\sum_{i=1}^{m_k}\xi_{k,i}\in[v,w])\ge
\int_v^wf_k(u)\,du$\,, provided  $0\le v\le w$\,, where 
$\xi_{k,1},\ldots,\xi_{k,m_k}$ are i.i.d. and are distributed as $\xi_k$\,.
 \end{enumerate}
Under
these hypotheses,
the   $Q(t)$ converge in distribution to random variable $\hat Q$\,, 
as $t\to\infty$\,,
 see  Down and Meyn \cite{MeyDow94}. 
The convergence holds for  arbitrary
initial vector $Q(0)$ and  the convergence rate  is
geometric for the metric of total variation. In addition,
if the random variables $Q(t)$
are augmented with residual service and interarrival times to produce a
Markov process, then that Markov process has a unique stationary
distribution, the distribution of $\hat Q$ being a marginal distribution.
Our main result is the following theorem.
\begin{theorem}
  \label{the:stat_LDP}
The sequence $\{\hat Q/n\,,n\in\N\}$
obeys the LDP in $\R_+^K$ for rate $n$ with the deviation function
 $\mathbf V(x)$\,.
\end{theorem}
\begin{remark}
  \label{re:nul'}
Under \eqref{eq:13},  there is no ''large deviation cost'' for staying at
the origin.
On taking in \eqref{eq:Psi}
$J=\mathcal{K}$\,, $y=0$\,, $\alpha=\lambda$\,,
$\varrho=P$\,, 
and $\delta=(I-P^T)^{-1}\lambda$ and noting that $\delta<\mu$ by \eqref{eq:13},
one can see by \eqref{eq:psi_J} that $\psi_{\mathcal{K}}(\alpha,\delta,\varrho)=0$\,, so 
$\Psi_{\mathcal{K}}(0)=0$ and $L(0,0)=0$\,. 
More generally,  $L(q(t),\dot q(t))=0$
when $q$ is ''a fluid limit queue length'' or the trajectory of the
law of large numbers, i.e.,
$\dot q(t)=\lambda+(P^T-I)\mu+(I-P^T)\phi(t)$\,, where $\phi(t)\in\R_+^K$
and $\phi_k(t)q_k(t)=0$\,, for $k\in\mathcal{K}$\,,
 cf. Puhalskii \cite{Puh07}.
The converse is also true: if $\Psi_J(y)=0$\,, then the infimum in
\eqref{eq:Psi} is attained at
$\alpha=\lambda$\,, $\delta_k=\mu_k$ when $k\not\in J$ and
$\varrho=P$\,. (For a proof, one notes that $\psi^A_k(\alpha_k)=0$\,,
$\psi^S_k(\delta_k)=0$\,, and $\psi^R_k(\varrho_k)=0$\,, if and only
if $\alpha_k=\lambda_k$\,, $\delta_k=\mu_k$\,, and $\varrho_k=(p_{k1},\ldots,p_{kK})$\,,
respectively.) As a byproduct, in \eqref{eq:4} $\mathbf I_0(q)$ can
be replaced with $\int_0^t L(q(s),\dot q(s))\,ds$\,.

\end{remark}
\section{Idempotent probability and the proof of Theorem \ref{the:stat_LDP}}
\label{sec:proof-theor-refth}

Let us  recap some notions of idempotent
probability, see, e.g., Puhalskii \cite{Puh01}.
Let $\Upsilon$ be a set. Function $\mathbf{\Pi}$ from the power
set of $\Upsilon$ to $[0,1]$ is called an idempotent probability
 if $\mathbf{\Pi}(\Gamma)=\sup_{\upsilon\in\Gamma}
\mathbf{\Pi}(\{\upsilon\}),\,\Gamma\subset \Upsilon$ and
$\mathbf{\Pi}(\Upsilon)=1$. The pair $(\Upsilon,\mathbf{\Pi})$ is
called an idempotent probability space.
For economy of notation, 
we denote $\mathbf{\Pi}(\upsilon)=\mathbf{\Pi}(\{\upsilon\})$.
Property $\mathcal{P}(\upsilon),\,\upsilon\in\Upsilon,$ 
pertaining to the elements of $\Upsilon$ is said to hold
$\mathbf{\Pi}$-a.e. if $\mathbf{\Pi}(\mathcal{P}(\upsilon)
\text{ does not hold})=0$\,, where, in accordance with a tradition of
probability theory, we define $\mathbf{\Pi}(\mathcal{P}(\upsilon)
\text{ does not hold})=\mathbf{\Pi}(\{\upsilon\in\Upsilon:\,
\mathcal{P}(\upsilon)
\text{ does not hold}\})$\,.
Function $f$ from  set
$\Upsilon$ equipped with idempotent probability $\mathbf{\Pi}$ to 
set $\Upsilon'$ 
is called an idempotent variable. 
 The idempotent distribution of the idempotent variable $f$ is defined
 as the set function $\mathbf{\Pi}\circ
f^{-1}(\Gamma)=\mathbf{\Pi}(f\in\Gamma),\,\Gamma\subset \Upsilon'$\,.
If $f$ is the canonical idempotent variable defined by
$f(\upsilon)=\upsilon$, then it has $\mathbf{\Pi}$ 
as the idempotent distribution. If
$f=(f_1,f_2)$\,, with $f_i$ assuming values in $\Upsilon'_i$\,,  then the (marginal)
distribution of $f_1$ is defined by
$\mathbf\Pi^{f_1}(\upsilon'_1)=\mathbf\Pi(f_1=\upsilon'_1)=
\sup_{\upsilon:\,f_1(\upsilon)=\upsilon'_1}\mathbf\Pi(\upsilon)$\,.
The  idempotent 
variables $f_1$ and $f_2$ are said to be independent if 
$\mathbf{\Pi}(f_1=\upsilon'_1,\,f_2=\upsilon'_2)=\mathbf{\Pi}(f_1=\upsilon'_1)
\mathbf{\Pi}(f_2=\upsilon'_2)$ for all $(\upsilon'_1,\upsilon'_2)\in
\Upsilon'_1\times\Upsilon'_2$\,,  
so, the joint distribution is the product of the
marginal ones.
Independence of finite collections of idempotent variables is defined
similarly. 
Collection $(X_t,\,t\in\R_+)$ 
of idempotent variables on $\Upsilon$
is called an idempotent process.
The functions $(X_t(\upsilon),\,t\in\R_+)$ 
for various $\upsilon\in\Upsilon$ are
called trajectories (or paths) of $X$.
Idempotent processes are said to be independent if they are
independent as idempotent variables with values in the associated
function spaces. The concepts of idempotent processes with independent
and (or) stationary increments mimic those for stochatic processes.

 If $\Upsilon$ is, in addition, a metric space and 
 the sets $\{\upsilon\in\Upsilon:\,\mathbf{\Pi}(\upsilon)\ge \kappa\}$ 
are compact for all
$\kappa\in(0,1]$, then $\mathbf{\Pi}$ is called a deviability.  
Obviously, $\mathbf{\Pi}$ is a deviability if and only if
$\mathbf{I}(\upsilon)=-\log\mathbf{\Pi}(\upsilon)$ is 
a deviation function.
 If $f$ is a continuous mapping from
$\Upsilon$ to another metric space $\Upsilon'$, then 
 $\mathbf{\Pi}\circ f^{-1}$ is a deviability on $\Upsilon'$.
As a matter of fact, for the latter property to hold, 
one can only require that $f$ be continuous on
the sets $\{\upsilon\in\Upsilon:\,\mathbf{\Pi}(\upsilon)\ge \kappa\}$ for
 $\kappa\in(0,1]$. In general, an idempotent variable is said to be
 Luzin if its idempotent distribution is a deviability.

Let  $\{\mathbf{P}_n,\,n\in\N\}$ be a  sequence of probability measures on
 metric space $\Upsilon$ endowed with the  Borel $\sigma$-algebra and
let $\mathbf{\Pi}$ be a deviability on $\Upsilon$.
The sequence $\{\mathbf{P}_n,\,n\in\N\}$
 is said to large deviation converge (LD converge) at rate $n$
 to  $\mathbf{\Pi}$ as $n\to\infty$ if $\lim_{n\to\infty}\Bl(\int_\Upsilon
 f(\upsilon)^{n}\,\mathbf{P}_n(d\upsilon)\Br)^{1/n}= 
\sup_{\upsilon\in \Upsilon}f(\upsilon)\mathbf{\Pi}(\upsilon)$
 for every bounded continuous $\R_+$-valued function $f$ on $\Upsilon$.
Equivalently, one may require that 
$\lim_{n\to\infty}\mathbf{P}_n(H)^{1/n}= 
\mathbf{\Pi}(H)$ for every $\mathbf\Pi$--continuity set $H$\,, which
is defined by the requirement that the values of $\mathbf\Pi$ on the
interior and closure of $H$ are equal to each other.
Obviously, the sequence $\{\mathbf{P}_n,\,n\in\N\}$ LD 
converges at rate $n$ to
 $\mathbf{\Pi}$ if and only if this sequence obeys the LDP for rate $n$ 
with deviation function
 $\mathbf{I}(\upsilon)=-\log\mathbf{\Pi}(\upsilon)$. 
Similarly, sequence $\mathbf\Pi_n$ of deviabilities on $\Upsilon$ is said to
converge weakly to deviability $\mathbf\Pi$\,, as $n\to\infty$\,, if
$\lim_{n\to\infty}\sup_{\upsilon\in \Upsilon}f(\upsilon)\mathbf{\Pi}_n(\upsilon)= 
\sup_{\upsilon\in \Upsilon}f(\upsilon)\mathbf{\Pi}(\upsilon)$
 for every bounded continuous $\R_+$-valued function $f$ on
 $\Upsilon$\,. The analogue of Prohorov's theorem holds: if the
 sequence $\mathbf\Pi_n$ is tight meaning that
$\inf_{\Gamma\in\Xi}\limsup_{n\to\infty}\mathbf\Pi_n(\Upsilon\setminus
\Gamma)=0$\,, where $\Xi$ represents the collection of compact
subsets of $\Upsilon$\,, then the $\mathbf\Pi_n$ converge to a
deviability along a subsequence.

LD convergence of probability measures 
can be also expressed as LD convergence in distribution of the
associated random variables to idempotent variables. 
 We say that 
  sequence
$\{X_n,\,n\in\N\}$  of random variables defined on respective probability spaces
$(\Omega_n,\mathcal{F}_n,\mathbf{P}_n)$ and assuming  values in 
$\Upsilon'$ 
LD converges in distribution at rate $n$ as $n\to\infty$
to   idempotent variable $X$ defined on 
idempotent probability space $(\Upsilon,\mathbf{\Pi})$ and assuming
values in $\Upsilon'$ if the sequence of the probability laws of the $X_n$ LD
converges to the idempotent distribution of $X$ at rate $n$. 
If  sequence $\{\mathbf{P}_n,\,n\in\N\}$
of probability measures on $\Upsilon$ LD converges to  deviability
$\mathbf{\Pi}$ on $\Upsilon$\,, then one has LD convergence in
distribution for the canonical setting.

We now return to the setting of generalised Jackson networks
 and 
let $\mathbf\Pi_x^Q(q)=e^{-\mathbf I_x(q)}$\,.
It is proved in Puhalskii \cite{Puh07} that under the hypotheses of
Theorem \ref{the:gen_jack}
there exists unique deviability $\mathbf\Pi_x$ on
$\Upsilon=\D(\R_+,\R_+^K\times\R_+^K\times\R_+^K
\times\R_+^K\times\R_+^K\times\R_+^{K\times  K})$ such that
the processes $\bl((Q^n(nt)/n\,,t\in\R_+),(B^n(nt)/n\,,t\in\R_+),
(D^n(nt)/n\,,t\in\R_+),(A^n(nt)/n\,,t\in\R_+),
(S^n(nt)/n,\,t\in\R_+), (R^n(nt)/n\,,t\in\R_+)\br)$ LD converge
at rate $n$  to the canonical
idempotent process $(q,b,d,a,s,r)$ on $\Upsilon$\,.
The component idempotent processes of 
$b$, $d$, $a$, $s$, and $r$ have $\mathbf{\Pi}_x$-a.e.
absolutely continuous nondecreasing trajectories starting at $0$, the
component idempotent processes of $b$ grow not faster than at rate $1$, and 
the component  idempotent processes of $q$ have $\mathbf{\Pi}_x$-a.e.
absolutely continuous  trajectories,
the idempotent process $q$ has idempotent distribution 
$\mathbf \Pi^Q_x$\,, the idempotent processes $a$, $s$
and $r$ are independent with respective idempotent distributions
$\mathbf{\Pi}^A$, $\mathbf{\Pi}^S$ and $\mathbf{\Pi}^R$ defined as
follows, where, by virtue of our working in a canonical setting,
identical  pieces of notation are used for denoting
idempotent processes and their sample trajectories:
\begin{align}
  \label{eq:24}
  \mathbf{\Pi}^A(a)=\prod_{k=1}^K\mathbf\Pi^A_k(a_k)\,,\;
\mathbf\Pi^A_k(a_k)=
\exp\Bl(-\int_0^\infty
\psi^A_k(\dot a_k(t))\,dt\Br)\,,\\
  \label{eq:24b}
  \mathbf{\Pi}^S(s)=\prod_{k=1}^K\mathbf\Pi^S_k(s_k)\,,\;
\mathbf\Pi^S_k(s_k)=
\exp\Bl(-\int_0^\infty\psi^S_k(\dot s_k(t))\,dt\Br)\,,\\
\label{eq:24d}
  \mathbf{\Pi}^R(r)=\prod_{k=1}^K\mathbf\Pi^R_k(r_k)\,,\;
\mathbf\Pi^R_k(r_k)=
\exp\Bl(-\int_0^\infty\psi^R_k(  \dot r_{k}(t))\,dt\Br)\,,
\end{align}
where
$a=(a(t)\,,t\in\R_+)=(a_1,\ldots,a_K)^T$\,, $a_k=(a_k(t)\,,t\in\R_+)$\,,
$a(t)=(a_1(t),\ldots,a_K(t))^T$\,,
$s=(s(t)\,,t\in\R_+)=(s_1,\ldots,s_K)^T$\,, $s_k=(s_k(t)\,,t\in\R_+)$\,,
$s(t)=(s_1(t),\ldots,s_K(t))^T$\,,
$r=(r(t)\,,t\in\R_+)=(r_{1},\ldots,r_K)^T$\,,
$r_k=(r_k(t)\,,t\in\R_+)=(r_{k1},\ldots,r_{kK})$\,,
$r_{kl}=(r_{kl}(t)\,,t\in\R_+)$\,,
$r_k(t)=(r_{k1}(t)\,,\ldots,r_{kK}(t))$\,,
$r(t)=(r_1(t),\ldots,r_K(t))^T$\,, the functions $a_k$\,, $s_k$\,, and
$r_{kl}$ being absolutely continuous 
with  $a_k(0)=0$, $s_k(0)=0$, $r_{kl}(0)=0$, $\dot{a}_k(t)\in\R_+$ a.e., 
$\dot{s}_k(t)\in\R_+$ a.e., and 
 $\dot{r}(t)\in\bbS^{K\times K}_+$  a.e.

Also
  $\mathbf{\Pi}_x$-a.e. the following equations hold for 
$t\in\R_+$ and $\,k\in\mathcal{K}$:
\begin{align}
  \label{eq:6a}
  q_k(t)&=x_k+a_k(t)+
\sum_{l=1}^Kr_{lk}\bl(d_{l}(t)\br)-d_k(t),\\
 \label{eq:7a}
  d_k(t)&=s_k\bl(b_k(t)\br),
\\
\label{eq:8a}
  \int_0^tq_k(u)\,db_k(u)&=\int_0^tq_k(u)\,du,
\end{align}
where 
$b=(b(t),\,t\in\R_+)$\,, $b(t)=(b_1(t),\ldots,b_K(t))^T$\,,
$d=(d(t),\,t\in\R_+)$\,, and $d(t)=(d_1(t),\ldots,d_K(t))^T$\,.
Equations \eqref{eq:6a} and \eqref{eq:7a} are obtained by taking 
large deviation limits in \eqref{eq:2} and \eqref{eq:1}, respectively.
%whereas, in order to derive \eqref{eq:8a}, one notes that \eqref{eq:5}
%implies that $
%  \int_0^tQ_k(u)\,dB_k(u)=\int_0^tQ_k(u)\,du$ and passes to the large
%  deviation limit in the latter equation.
 It is noteworthy that since in \eqref{eq:24}, \eqref{eq:24b} and \eqref{eq:24d} the sample
trajectories enter the deviabilities only through their derivatives,
the idempotent processes $a$\,, $s$ and $r$ have independent and stationary
increments.

By \eqref{eq:6a}, $q(0)=x$ $\mathbf\Pi_x$-a.e.
In order to allow the initial value $q(0)$ to have a nondegenerate
idempotent distribution, we introduce
\begin{equation}
  \label{eq:27}
  \mathbf\Pi(\upsilon)=\sup_{x\in\R_+^K}\mathbf\Pi_x(\upsilon)\tilde{\mathbf\Pi}^{Q_0}(x)\,,
\end{equation}
where $\tilde{\mathbf\Pi}^{Q_0}$ is a deviabiltiy on $\R_+^K$\,.
One can see that $\mathbf\Pi$ is a deviability on
$\Upsilon$\,. 
Obviously, $\mathbf\Pi(q(0)=x)=\tilde{\mathbf\Pi}^{Q_0}(x)$
and $q(0)$\,, $a$\,, $s$ and $r$ are independent under $\mathbf\Pi$\,.
Also, the marginal idempotent distribution of $q$ is given by
\begin{equation*}
  \mathbf\Pi^Q(q)=\sup_{x\in\R_+^K}\mathbf\Pi^Q_x(q) \tilde{\mathbf\Pi}^{Q_0}(x)\,.
\end{equation*}
By \eqref{eq:6a}, $\mathbf\Pi$--a.e.,
\begin{equation}
  \label{eq:11}
  q_k(u)=q_k(0)+a_k(u)+\sum_{l=1}^Kr_{lk}(d_l(u))-d_k(u)\,.
\end{equation}
Let 
\begin{equation*}
  \mathbf\Pi_{x,t}(y)=
\sup_{q: q(t)=y}\mathbf \Pi^Q_x(q)\,, \;
\mathbf \Pi_{x,t}(\Gamma)=\sup_{y\in \Gamma}\mathbf\Pi^Q_{x,t}(y)\,,
\text{ where }\Gamma\subset\R_+^K\,.
\end{equation*}
The definition implies the semigroup property that 
\begin{equation}
  \label{eq:30}
  \mathbf\Pi_{x,u+v}(y)=\sup_{z\in\R_+^K}\mathbf\Pi_{x,u}(z)\mathbf\Pi_{z,v}(y)\,.
\end{equation}
For $\Delta\subset\mathcal{K}$\,, we will
denote by $\mathbf1_\Delta$ the vector with unity entries whose 
 dimension equals the number
of elements in $\Delta$\,. For compactness of
notation, we let $\mathbf1=\mathbf1_\mathcal{K}$\,.
\begin{lemma}
  \label{le:renewal_deviab}
Given $\kappa>0$ and  $\epsilon>0$\,, there exists $T>0$ such that
\begin{multline*}
\mathbf \Pi\bl(
\cup_{u\ge T}\{\{a(u)\notin[(\lambda-\epsilon\mathbf1) u,(\lambda+
\epsilon\mathbf1)u]\}\cup\{
s(u)\notin[(\mu-\epsilon\mathbf1) u,(\mu+\epsilon\mathbf1)u]\}\\
\cup\{r(u)\notin[(P-\epsilon I) u,(P+\epsilon I)u]\}\}
\br)<\kappa\,.
\end{multline*}
\end{lemma}
\begin{proof}
By the maxitivity property
that $\mathbf\Pi(\cup_i\Gamma_i)=\sup_{i}\mathbf\Pi(\Gamma_i)$\,,
for arbitrary collection of sets $\Gamma_i$\,,
 it suffices
 to work with $\mathbf \Pi(
a(u)\notin[(\lambda-\epsilon\mathbf1) u,(\lambda+\epsilon\mathbf1)u])$ only.
By the LD convergence in distribution 
of $(A(nt)/n\,,t\in\R_+)$ to $a$ and  Lemma \ref{le:ren} relegated to
the appendix, whose assertion can be 
found in Appendix A of Bell and Williams \cite{BelWil01}, for some $\sigma\in(0,1)$\,,
\begin{multline*}
  \mathbf \Pi(
a(u)\notin[(\lambda-\epsilon\mathbf1) u,(\lambda+\epsilon\mathbf1)u])=
\mathbf\Pi^A(a(u)\notin[(\lambda-\epsilon\mathbf1) u,(\lambda+\epsilon\mathbf1)u])\\\le
\liminf_{n\to\infty}\mathbf P\bl(\abs{\frac{A(nu)}{n}-\lambda u}>\epsilon 
u\br)^{1/n}\le\sigma^u\,.
\end{multline*}
\end{proof}
\begin{lemma}
\label{le:nul}
  Given bounded set $B\subset\R_+^K$ and $\kappa\in(0,1)$\,, there exists
 $\hat T>0$ such that
if $\mathbf\Pi^Q(q)> \kappa$ and $q(0)\in B$\,, then 
$\min_{u\in[0,\hat T]}\abs{q(u)}=0\,.$
\end{lemma}
\begin{proof}
The proof proceeds by establishing, initially,  that
in the long run the
idempotent processes 
$a(t)$\,, $s(t)$ and $r(t)$ ''with great deviability'' stay close to
the corresponding fluid trajectories $\lambda t$\,, $\mu t$ and
$Pt$\,, respectively. Then, drawing on the proof  of the stability of
fluid models of queueing networks in
Bramson \cite{Bra06,Bra08},
it is shown that owing to condition \eqref{eq:13}  the function
$\mathbf1\cdot(I-P^T-\epsilon I)^{-1}q(u)$ 
decreases linearly with $u$\,,  provided $\epsilon$ is small enough,
which implies that the function must attain
$0$\,.

By \eqref{eq:13}, there exists
$\epsilon>0$ such that
$(I-P^T-\epsilon I)^{-1}(\lambda+\epsilon\mathbf1)\le\mu-\epsilon\mathbf1$
and $(I-P^T+\epsilon
I)^{-1}(\lambda-\epsilon\mathbf1)\le\mu-\epsilon\mathbf1$\,.
(In the course of the proof, potentially smaller $\epsilon$ will be
needed. Yet, there exists $\epsilon$ that satisfies all the
requirements.
Importantly, it depends neither on $\kappa$ nor on $B$\,.)
By Lemma \ref{le:renewal_deviab}, there exists $T>0$ such that
$\mathbf \Pi(a(u)\ge(\lambda+\epsilon\mathbf1)u \text{ for some }u\ge
T)<\kappa$\,,
$\mathbf \Pi(a(u)\le(\lambda-\epsilon\mathbf1)u \text{ for some }u\ge T)
<\kappa$\,,
$\mathbf \Pi(s(u)\le(\mu-\epsilon\mathbf1)u \text{ for some }u\ge
T)<\kappa$\,,
$\mathbf \Pi(s(u)\ge(\mu+\epsilon\mathbf1)u \text{ for some }u\ge T)<\kappa$\,,
$\mathbf \Pi(r(u)\ge(P+\epsilon I)u\text{ for some }u\ge
T)<\kappa$\,, 
and 
$\mathbf \Pi(r(u)\le(P-\epsilon I)u\text{ for some }u\ge T)<\kappa$\,.

Let
\begin{multline*}
  \Gamma_\kappa=\{\upsilon:\,
(\lambda-\epsilon\mathbf1)u\le a(u)\le(\lambda+\epsilon\mathbf1)u\,,
(P-\epsilon I)u\le r(u)\le(P+\epsilon I)u\,,\\
\text{ and }
(\mu-\epsilon\mathbf1)u\le s(u)\le (\mu+\epsilon\mathbf1)u\,,
\text{ for }u\ge T\}.
\end{multline*}
We have that $\mathbf\Pi(\Gamma_\kappa^c)<\kappa$ and that
on $\Gamma_\kappa$\,,
provided $u\ge T$\,, by \eqref{eq:6a}, %%%%% and \eqref{eq:7a}, %and \eqref{eq:8a},
\begin{align}
  \label{eq:16}
  q(0)+(\lambda-\epsilon\mathbf1)u+(P^T-\epsilon I)d(u)-d(u)\le q(u)
\le q(0)+(\lambda+\epsilon\mathbf1)u+(P^T+\epsilon I)d(u)-d(u)\,.
\end{align}

%Hence,
%\begin{align}
%  \label{eq:17}
%(I- (P^T-\epsilon I))^{-1} q(0)+(I- (P^T-\epsilon I))^{-1}
%(\lambda-\epsilon)u-d(u)\le
%(I- (P^T-\epsilon I))^{-1} q(u)\,,\\
%(I- (P^T+\epsilon I))^{-1} q(u)\le(I- (P^T+\epsilon I))^{-1} q(0)+
%(I- (P^T+\epsilon I))^{-1}(\lambda+\epsilon)u-d(u)\,.
%\end{align}
Let us show that there exists $S\ge 2T$ 
 such that $b_k(S)\ge T$ for all $k$ on $\Gamma_\kappa$\,.
Intuitively, this is the case because otherwise some
$s_k(b_k(u))$ would be ''bounded'' whereas $a_k(u)$ can be arbitrarily
great for great $u$ pushing  $b_k(u)$ past $T$\,. 
Formally, assuming that $\epsilon<\lambda_k$\,, for all $k$\,,
 let $S\ge 2T$ be such that
$(\lambda_k-\epsilon)(S-T)-(\mu_k+\epsilon)T>0$\,, for all $k$\,.
If $b_k(S)< T$\,, for some $k$\,, then, by \eqref{eq:7a},
 $d_k(u)\le s_k(T)$ for 
$u\in[0,S]$\,.
By \eqref{eq:11}, on $\Gamma_\kappa$\,,
for $u\in[ S-T,S]$\,,
$q_k(u)\ge a_k(u)-s_k(T)\ge (\lambda_k-\epsilon)u-(\mu_k+\epsilon)T
>0$\,. 
Therefore, by \eqref{eq:8a},
$\dot b_k(u)=1$ a.e. when $u\in[S-T, S]$\,, so  $b_k(S)\ge T$\,, 
which contradicts the
assumption that $b_k(S)< T$\,. It is worth noting that
whereas both
$T$ and $S$ may depend on either $\epsilon$ or $\kappa$\,,
 neither of them  depends on $q(0)$\,.

We now assume that $q$ is piecewise linear, which assumption is to be
disposed of later.
Let us suppose that $q_k(u)>0$  for $u$ in a right neighborhood of $S$
  for some $k$  on $\Gamma_\kappa$\,.
Then,  $b_k(u)=b_k(S)+u-S\ge T+ u-S$\,, for $u\ge S$\,,
until $q_k(u)$ hits zero.
Accordingly, $d_k(u)=s_k(b_k(u))\ge (\mu_k-\epsilon)(u-S)$\,.  Hence,
if $q(u)>0$ entrywise in a right neighborhood of $S$\,, then
\begin{multline*}
  d(u)\ge (I- (P^T+\epsilon I))^{-1}(\lambda+\epsilon\mathbf1)(u-S)=
\bl(\nu+\epsilon(I- (P^T+\epsilon I))^{-1}(\nu+\mathbf1)\br)(u-S)\,,
\end{multline*}
where we denote
\begin{equation}
  \label{eq:15}
  \nu=(I-P^T)^{-1}\lambda\,.
\end{equation}
As a consequence, for some $\rho>0$\,, which is dependent on
$\epsilon$ only,
while $q(u)>0$ entrywise,
\begin{equation}
  \label{eq:22}
  d(u)\ge(\nu+\rho\mathbf1)(u-S)\,.
\end{equation}
By the righthand inequality in \eqref{eq:16} and \eqref{eq:15},  
\begin{multline}
  \label{eq:8}
    \mathbf1\cdot(I-P^T-\epsilon I)^{-1}q(u)\le
  \mathbf1\cdot(I-P^T-\epsilon I)^{-1}q(0)
+
\mathbf1\cdot(I-P^T-\epsilon I)^{-1}(\lambda+\epsilon \mathbf1)u-\mathbf1\cdot d(u)\\
%=\mathbf1\cdot(I-P^T-\epsilon I)^{-1}q(0)
%+  \mathbf1\cdot(I-P^T-\epsilon I)^{-1}
%(I-P^T-\epsilon I+\epsilon I)(I-P^T)^{-1}(\lambda+\epsilon \mathbf1)u-\mathbf1\cdot d(%u)\\
%=\mathbf1\cdot(I-P^T-\epsilon I)^{-1}q(0)
%+\mathbf1\cdot(I-P^T)^{-1}(\lambda+\epsilon \mathbf1)u +
%\epsilon \mathbf1\cdot(I-P^T-\epsilon I)^{-1}(I-P^T)^{-1}(\lambda+\epsilon \mathbf1)u-\mathbf1\cdot d(u)\\
=\mathbf1\cdot(I-P^T-\epsilon I)^{-1}q(0)
+\mathbf1\cdot(\nu u-d(u))+\epsilon \mathbf1\cdot(I-P^T)^{-1}\mathbf1 u +
\epsilon \mathbf1\cdot(I-P^T-\epsilon I)^{-1}(I-P^T)^{-1}(\lambda+\epsilon \mathbf1)u\,.
\end{multline}
By \eqref{eq:22}, there exist $\hat\rho>0$ and $\gamma>0$ such that, provided
$\epsilon$ is small enough, if $ u\ge S$\,, then, while
 $q(u)$ stays entrywise positive,
\begin{equation}
  \label{eq:7}
  \mathbf1\cdot(I-P^T-\epsilon I)^{-1}q(u)\le
\mathbf1\cdot(I-P^T-\epsilon I)^{-1}q(0)+\gamma S-\hat\rho u\,.
\end{equation}
  
Let us show that similar inequalities hold on $\Gamma_\kappa$ for all $u\ge S$\,.
Given $v\ge S$\,,
let $O$ denote a possibly empty set of indices $k$ such that 
 $q_k(u)=0$ on some interval $[v,v+\eta]$ 
and $q_k(u)>0$ if $k\notin O$ and $u\in(v,v+\eta)$\,.
Such  $\eta$ exists because $q(u)$ is piecewise linear.
We assume that $q(u)\not=0$ on $[v,v+\eta]$\,, so
 $O$ is a proper subset of $\mathcal{K}$\,.
By  the lefthand inequality in  \eqref{eq:16}, on $\Gamma_\kappa$\,,
\begin{equation}
  \label{eq:20}
  q_k(0)+(\lambda_k-\epsilon)u+((P^T-\epsilon I)d(u))_k-d_k(u)\le0\,,
  \text{ for } k\in O\,, u\in[v,v+\eta]\,.
\end{equation}
Therefore, using subscript $O$ and $O^c$ to denote restrictions of vectors to
indices in $O$ and $O^c$ respectively, 
and using subscripts $OO$ and $OO^c$ to denote restrictions of
matrices to entries with both indices in $OO$ and $OO^c$\,, respectively, we have that
\begin{equation*}
  d_O(u)\ge q_O(0)+(\lambda_O-\epsilon\mathbf 1_O)u
+(P^T-\epsilon I)_{OO}d_O(u)+(P^T-\epsilon I)_{OO^c}d_{O^c}(u)\,,
\end{equation*}
so, assuming $\epsilon$ is small enough,
\begin{equation}
  \label{eq:26}
    d_O(u)\ge(I-(P^T-\epsilon I))_{OO}^{-1}
\bl(q_O(0)+(\lambda_O-\epsilon\mathbf 1_O)u+(P^T-\epsilon I)_{OO^c}d_{O^c}(u)\br)\,.
\end{equation}

On the other hand, by \eqref{eq:15},
$  \lambda_O=(I-P^T)_{OO}\nu_O-P^T_{OO^c}\nu_{O^c}\,.
$ Substitution in \eqref{eq:26} and rearranging yield
\begin{multline*}
    d_O(u)%\ge(I-(P^T-\epsilon I))_{OO}^{-1}
%\bl(q_O(0)+((I-P^T)_{OO}\nu_O-P^T_{OO^c}\nu_{O^c}-\epsilon\mathbf1_O)u+(P^T-\epsilon
%I)_{OO^c}d_{O^c}(u)\br)
%\ge (I-(P^T-\epsilon I))_{OO}^{-1}
%(I-P^T)_{OO}\nu_Ou
%+(I-(P^T-\epsilon I))_{OO}^{-1}P^T_{OO^c}(d_{O^c}(u)-\nu_{O^c}u)
%\\-\epsilon(I-(P^T-\epsilon I))_{OO}^{-1}(\mathbf1_Ou+
%I_{OO^c}d_{O^c}(u))\\
\ge\nu_Ou
+(I-(P^T-\epsilon I))_{OO}^{-1}
\bl(P^T_{OO^c}(d_{O^c}(u)-\nu_{O^c}u)-\epsilon\nu_Ou
-\epsilon\mathbf1_Ou-\epsilon
I_{OO^c}d_{O^c}(u)\br))\,.
\end{multline*}
In analogy with the derivation of \eqref{eq:22}, one obtains that, for
some $\rho_O>0$\,,
\begin{equation}
  \label{eq:23}
  d_{O^c}(u)\ge (\nu_{O^c}+\rho_{O}\mathbf1_{O^c})(u-S)\,.
\end{equation}
Therefore,
for $u\in[v,v+\eta]$\,,
\begin{multline}
  \label{eq:21}
d_O(u)\ge \nu_Ou
+(I-(P^T-\epsilon I))_{OO}^{-1}
\bl(-P^T_{OO^c}(\nu_{O^c}+\rho_{O}\mathbf1_{O^c})S-\epsilon\nu_Ou
-\epsilon\mathbf1_Ou-\epsilon
I_{OO^c}d_{O^c}(u)\br)\,.
\end{multline}
Since $O^c\not=\emptyset$\,, by \eqref{eq:23}, \eqref{eq:21} and the
bound $d(u)\le(\mu+\epsilon)u$ when $u\ge S$\,, 
 there exist $\tilde\rho>0$ and $\gamma>0$ which do not depend on $O$ 
 such that, 
assuming  $\epsilon$ is small enough, for $u\in[v,v+\eta]$\,,
\begin{equation*}
  \mathbf1\cdot d(u)=\mathbf1_O\cdot d_O(u)+\mathbf1_{O^c}\cdot d_{O^c}(u)
\ge \mathbf1\cdot\nu u+\tilde\rho u-\gamma S\,.
\end{equation*}
By \eqref{eq:8}, we obtain that  \eqref{eq:7} still holds,
 for suitable $\hat\rho>0$ which does not depend on $O$ and $u\in[v,v+\eta]$\,,
provided $\epsilon$ is small enough.
We can repeat the same argument over and over again, so, \eqref{eq:7}
holds until $q(u)=0$\,. Hence, one can take
\begin{equation*}
  \hat T=\frac{1}{\hat\rho}\,(
\abs{(I-P^T-\epsilon I)^{-1}\mathbf1}\sup_{x\in B}\abs{x}+\gamma S)
\end{equation*}
as the time by which $q$ is bound to hit the origin.

Suppose now that $q$ is not necessarily piecewise linear and
$\mathbf\Pi^Q(q)> \kappa$\,. 
By 
Lemmas 4.1--4.4 in  Puhalskii \cite{Puh07}, there exist
 piecewise linear  $q^n$ which converge
to $q$ as $n\to\infty$
 such that $\mathbf\Pi^Q(q^n)\to\mathbf\Pi^Q(q)$\,. By what's
been proved,
there exist $t^n$ from $[0,\hat T]$ such that $\abs{q^n(t^n)}=0$\,.
Since $\abs{q^n(t^n)-q(t^n)}\to0$\,, it follows that $\abs{q(t')}=0$ where
$t'$ represents a subsequential limit of the $t^n$\,.
\end{proof}
\begin{theorem}
  \label{le:predel}
There exists deviability $\hat{\mathbf \Pi}$ on $\R_+^K$ such
that, for every bounded set  $B\subset \R_+^K$\,, 
\begin{equation}
  \label{eq:9}
  \lim_{t\to\infty}\sup_{x\in B}\sup_{y\in\R^K_+}\abs{
\mathbf\Pi_{x,t}(y)-\hat{\mathbf\Pi}(y)}=0\,.
\end{equation}
Furthermore, given $y$\,, $\mathbf\Pi_{x,t}(y)=\hat{\mathbf\Pi}(y)$ for all
$t$ great enough and all $x\in B$\,.
The deviability $\hat{\mathbf\Pi}$ is a unique stationary  deviability for
the semigroup $\mathbf\Pi_{x,t}$ meaning that, for all $y\in\R_+^K$
and $t\in\R_+$\,,
\begin{equation*}
  \hat{\mathbf\Pi}(y)=\sup_{x\in\R_+^K}\mathbf\Pi_{x,t}(y)\hat{\mathbf\Pi}(x)\,.
\end{equation*}
\end{theorem}
\begin{proof}
  One can see  that $\mathbf\Pi_{0,t}(y)$ is a
nondecreasing function of $t$\,. Indeed, let $u\le t$\,.
Given function $q$  such that
$q(0)=0$ and $q(u)=y$\,, we can associate with it function $\tilde q$
such that $\tilde q(v)=0$ for $v\in[0, t-u]$ and $\tilde q(v)=q(v-(t-u))$
for $v\in[u,t]$\,.  It follows that
$\int_0^{t-s}L(\tilde q(r),\dot{\tilde q}(r))\,dr=0$\,, so 
$\mathbf I_0(\tilde q)=\mathbf I_0(q)$ yielding the desired monotonicity.
We let
\begin{equation}
  \label{eq:10}
  \hat{\mathbf \Pi}(y)=\lim_{t\to\infty}\mathbf\Pi_{0,t}(y)\,.
\end{equation}
Let us show that
 $\mathbf\Pi_{0,t}(y)$ levels off eventually as a
function of $t$\,. 
Let $\kappa>0$\,.
We define $ t'$ as $\hat T$ in the statement of Lemma \ref{le:nul}
with  $\{x\in\R_+^K:\,\abs{x}\le1\}\, $ as set $B$\,.
Suppose that $\mathbf\Pi_{0,t}(y)>\kappa$\,, 
where $t\ge t'+1$\,. Let
$q$ be such that $q(0)=0$\,, $q(t)=y$ and
   $\mathbf \Pi_0^Q(q)=\mathbf\Pi_{0,t}(y)$\,.
 Let $\tilde
t=\inf_{t:\,q(t)\ge1}\wedge1$\,. Then $q(\tilde t)\le1$ and $0<\tilde
t\le1$\,. By Lemma \ref{le:nul}, there exists
$\breve t\in[\tilde t, t'+1]$ such that $q(\breve t)=0$\,. On defining $\tilde
q(s)=q(s+\breve t)$\,, we have that $\mathbf\Pi_0^Q(\tilde q)\ge
\mathbf\Pi^Q_0( q)$\,. On the other hand, since $ \breve t\le t$\,,
we have that $\tilde q(t-\breve t)=y$
which implies  that
$\mathbf\Pi_0^Q(\tilde q)\le \mathbf\Pi_{0,t-\breve t}(y)\le
\mathbf \Pi_{0,t}(y)=\mathbf\Pi_0^Q(q)$\,, so, $q(u)=0$ on
$[0,\breve t]$\,, for Remark \ref{re:nul'} implies that if
$q(0)=0$ and $q(u)=0$ for some $u>0$\,, then
$\int_0^u L(q(v),\dot q(v))\,dv=0$ if and only if 
$q(v)=0$ on $[0,u]$\,. Hence, $\breve t=\tilde t=1$\,, so,
 $\mathbf\Pi_{0,t-1}(y)=\mathbf\Pi_{0,t}(y)$\,. This proves that if 
$\mathbf\Pi_{0,t}(y)>\kappa$ and $t\ge t'+1$\,, then 
$\mathbf\Pi_{0,t'+1}(y)=\mathbf\Pi_{0,t}(y)$\,.
We also have that 
$\mathbf\Pi_{0,t}(y)\le \kappa\vee\mathbf\Pi_{0,t'+1}(y)$\,, for all
$t$
and $y$\,. Hence,
 the net of deviabilities $\mathbf\Pi_{0,t}$ is tight, so,
$\hat{\mathbf \Pi}$ is a deviability too.

Let us prove that
\begin{equation}
  \label{eq:28}
  \lim_{t\to\infty}\sup_{x\in B}\sup_{y\in\R^K_+}\abs{
\mathbf\Pi_{x,t}(y)-\mathbf\Pi_{0,t}(y)}=0\,.
\end{equation}
A coupling argument is employed.
We prove, at first, that, for arbitrary $\kappa>0$\,, 
\begin{equation}
  \label{eq:31}
  \mathbf\Pi_{x,t}(y)\le  \mathbf\Pi_{0,t}
(y)+\kappa\text{ for all } x\in B \text{ and } 
y\in \R^K\text{ provided $t$ is great enough}\,.
\end{equation}
By Lemma \ref{le:nul}, there exists $\hat T$ such that 
if  $q(0)\in B$ and 
$\mathbf\Pi^Q(q)> \kappa$\,, then $q(u)=0$ for some $u\in[0,\hat T]$\,.
Let us fix $x\in B$ and $y\in\R_+^K$\,.
One can assume that $t\ge \hat T$ and that $\mathbf\Pi_{x,t}(y)>\kappa$\,.
Let trajectory $\hat q$ be such that $\hat q(0)=x$\,, $\hat q(t)=y$ and
$  \mathbf\Pi_{x,t}(y)=\mathbf\Pi^Q(\hat q)\,$.
By Lemma \ref{le:nul}, there exists $\hat T_1\in[0,\hat T]$  such that 
$\hat q(\hat T_{1})=0$\,.
We define $\tilde q$ by letting $\tilde q(u)=0$ when $u\le \hat
T_{1}$ and $\tilde q(u)=\hat q(u)$ when $u\ge \hat T_{1}$\,.
By Remark \ref{re:nul'},
$  \mathbf\Pi^Q(\hat q)\le \mathbf\Pi^Q(\tilde q)\le \mathbf\Pi_{0,t}
(y)\,,
$ proving \eqref{eq:31}.

On the other hand, given  $t\ge \hat T+1$\,, $x\in B$\,,
and $ q$ such that $ q(0)=0$\,, $
q(t)=y$\,, 
$\mathbf\Pi^Q( q)=\mathbf\Pi_{0,t}(y)> \kappa$\,,
and $ q(u)=0$ for all $u\in[0,\hat T]$ (the latter can be always
assumed as we have seen),
 we define 
$\hat q$ with $\hat q(0)=x$ 
by letting it follow the law of large numbers until it hits
zero at some $\hat T_1\in[0,\hat T]$ and by letting 
$\hat q(u)= q(u-\hat T_1)$\,, for $u\ge \hat T_1$\,.
Since $\mathbf\Pi^Q(\hat q)=\mathbf\Pi^Q( q)$ by Remark \ref{re:nul'}\,, we obtain that
$\mathbf\Pi_{0,t}(y)= \mathbf\Pi^Q(\hat q)\le \mathbf \Pi_{x,t}(y)$\,,
which concludes the proof of \eqref{eq:28}.

We have shown that $\mathbf\Pi_{x,t}(y)\to\hat{\mathbf\Pi}(y)$\,, as
$t\to\infty$\,, uniformly over $y\in\R_+^K$ and 
over $x$ from bounded sets. 
It follows that, for arbitrary initial deviability $\tilde{\mathbf\Pi}^{Q_0}$\,,
\begin{equation*}
  \lim_{t\to\infty}\sup_{y\in\R_+^K}\abs{\mathbf\Pi(q(t)=y)-\hat{\mathbf\Pi}(y)}=0\,.
\end{equation*}
Letting $u\to\infty$ in \eqref{eq:30} implies that $\hat{\mathbf\Pi}$
is a unique stationary initial deviability.
(For, if $\mathbf\Pi'$ is another stationary deviability, then
$\abs{\hat{\mathbf\Pi}(y)-\mathbf\Pi'(y)}=
\abs{\hat{\mathbf\Pi}(y)\sup_{x\in\R_+^K}\mathbf\Pi'(x)-
\sup_{x\in\R_+^K}\mathbf\Pi'(x)\mathbf\Pi_{x,t}(y)}\le
\sup_{x\in\R_+^K}\abs{\mathbf\Pi'(x)\hat{\mathbf\Pi}(y)
-\mathbf\Pi'(x)\mathbf\Pi_{x,t}(y)}
\le 
\sup_{\substack{x\in\R_+^K\,:\\\mathbf\Pi'(x)\ge\kappa}}\abs{\hat{\mathbf\Pi}(y)
-\mathbf\Pi_{x,t}(y)}\vee\kappa$\,, where
$\kappa\in(0,1]$\,, and one can let $t\to\infty$\,.)

\end{proof}
\begin{remark}
  The proof shows that the value of $t$ where the
  $\mathbf\Pi_{0,t}(y)$ level off can be chosen uniformly over $y$
  such that $\hat{\mathbf\Pi}(y)\ge\kappa$\,.
\end{remark}
\begin{proof}[Proof of Theorem \ref{the:stat_LDP}]
Let $\mathbf Q^n$ denote the distribution of $\hat Q/n$ and let
$\mathbf Q^n_{0,t}$
denote the distribution of $Q(nt)/n$\, for $Q(0)=0$\,. Let
 $H\subset \R^K$ be a $\hat{\mathbf\Pi}$--continuity set.
We have that 
\begin{equation}
  \label{eq:6}
\abs{ \mathbf Q^n(H)^{1/n}-\hat{\mathbf \Pi}(H)}\le
\abs{\mathbf Q^n(H)-\mathbf Q_{0,t}^n(H)}^{1/n}+
 \abs{\mathbf Q_{0,t}^n(H)^{1/n}-\mathbf \Pi_{0,t}(H)}
+\abs{\mathbf \Pi_{0,t}(H)-\hat{\mathbf \Pi}(H)}\,.
\end{equation}
By 
Theorem 4.1 in Down and Meyn \cite{MeyDow94},
 there exist $A>1$ and
$\rho\in(0,1)$ such that
$  \abs{\mathbf Q^n(H)-\mathbf Q^n_{0,t}(H)}
\le A\rho^{ nt}\,.
$ Given $\epsilon>0$\,, let
 $t$ be such that
$A\rho^{ t}<\epsilon$ and $\abs{\mathbf \Pi_{0,t}(H)-
\hat{\mathbf \Pi}(H)}<\epsilon$\,.
Since, by Theorem \ref{the:gen_jack}, for all $n$ great enough,
$
\abs{\mathbf Q_{0,t}^n(H)^{1/n}-\mathbf \Pi_{0,t}(H)}<\epsilon$\,, it
follows that
$\abs{\mathbf
Q^n(H)^{1/n}-\hat{\mathbf \Pi}(H)}<3\epsilon$\,, for all $n$ great enough.
(Alternatively, one may let $n\to\infty$ and then let
$t\to\infty$ in \eqref{eq:6}.)
Finally, $\hat{\mathbf\Pi}(x)=e^{-\mathbf V(x)}$ by \eqref{eq:4} and \eqref{eq:10}.
\end{proof}
\begin{remark}
  Since $\mathbf\Pi_{0,t}(H)\uparrow \hat{\mathbf \Pi}(H)$\,, as
  $t\to\infty$\,, one can
  see by \eqref{eq:6}, that, more generally, geometric ergodicity of $\mathbf
  Q_{0,t}$\,, as $t\to\infty$\,, for the metric of total variation and
  a sample path LDP for $(Q_{nt}/n\,,t\ge0)$ with $Q_0=0$\,, imply an LDP for
  $\mathbf Q^n$\,.
\end{remark}
\appendix
\section{Appendix}
\label{sec:app}

\begin{lemma}
\label{le:ren}  Let $(N(t)\,,t\in\R_+)$ be a renewal process with rate $\lambda$\,. Suppose
  that certain exponential moments
of the inter-renewal times are 
  finite.
Then, given arbitrary $\epsilon>0$\,, there exists $\sigma\in(0,1)$
such that, for all $t\in\R_+$\,,
\begin{equation*}
  \limsup_{n\to\infty}
\mathbf P(\abs{\frac{N_{nt}}{n}-\lambda t}>\epsilon t)^{1/n}\le\sigma^t\,.
\end{equation*}
\end{lemma}
\begin{proof}
  Let $\vartheta_1,\vartheta_2,\ldots$ denote the successive inter-renewal times.
For suitable $\alpha>0$\,,
\begin{multline*}
 \mathbf P(\abs{\frac{N_{nt}}{n}-\lambda t}>\epsilon t)\le
\mathbf P\bl(\sum_{i=1}^{\lfloor
  n(\lambda+\epsilon)t\rfloor}(\vartheta_i-\frac{1}{\lambda})
\le
nt-
\frac{\lfloor n(\lambda+\epsilon)t\rfloor}{\lambda}\br)\\+\mathbf P\bl(
\sum_{i=1}^{\lfloor
  n(\lambda-\epsilon)t\rfloor}(\vartheta_i-\frac{1}{\lambda})
\ge nt-\frac{\lfloor n(\lambda-\epsilon)t\rfloor}{\lambda}\br)
\le \exp\bl(\lfloor
  n(\lambda+\epsilon)t\rfloor\ln\mathbf E\exp(-\alpha
  (\vartheta_1-\frac{1}{\lambda}))\\
-\alpha(\frac{\lfloor n(\lambda+\epsilon)t\rfloor}{\lambda}-nt)\br)
+\exp\bl(\lfloor
  n(\lambda-\epsilon)t\rfloor\ln\mathbf E\exp(\alpha
  (\vartheta_1-\frac{1}{\lambda}))
-\alpha(nt-\frac{\lfloor n(\lambda-\epsilon)t\rfloor}{\lambda})\br)\,.
\end{multline*}
Hence,
\begin{multline*}
  \limsup_{n\to\infty}
 \mathbf P(\abs{\frac{N_{nt}}{n}-\lambda t}>\epsilon t)^{1/(nt)}\le
\exp \bl(
  (\lambda+\epsilon)\ln\mathbf E\exp(-\alpha
  (\vartheta_1-\frac{1}{\lambda}))
-\alpha\frac{\epsilon}{\lambda}\br)\\
\vee\exp \bl(
  (\lambda-\epsilon)\ln\mathbf E\exp(\alpha
  (\vartheta_1-\frac{1}{\lambda}))
-\alpha\frac{\epsilon}{\lambda}\br)\,.
\end{multline*}
Since $\mathbf E(\vartheta_1-1/\lambda)=0$\,, the latter righthand side is
less than unity for $\alpha$ small enough.
\end{proof}
% \bibliographystyle{plain}
%\bibliography{large,idemp,puh,stoch,sprav,optim,que}
\def\cprime{$'$} \def\cprime{$'$} \def\cprime{$'$} \def\cprime{$'$}
  \def\cprime{$'$} \def\polhk#1{\setbox0=\hbox{#1}{\ooalign{\hidewidth
  \lower1.5ex\hbox{`}\hidewidth\crcr\unhbox0}}} \def\cprime{$'$}
  \def\cprime{$'$} \def\cprime{$'$} \def\cprime{$'$} \def\cprime{$'$}
  \def\cprime{$'$}

\def\cprime{$'$} \def\cprime{$'$} \def\cprime{$'$} \def\cprime{$'$}
  \def\cprime{$'$} \def\polhk#1{\setbox0=\hbox{#1}{\ooalign{\hidewidth
  \lower1.5ex\hbox{`}\hidewidth\crcr\unhbox0}}} \def\cprime{$'$}
  \def\cprime{$'$} \def\cprime{$'$} \def\cprime{$'$} \def\cprime{$'$}
  \def\cprime{$'$}

% \bibliographystyle{plain}
%\bibliography{large,idemp,puh,stoch,sprav,optim,que}

\end{document}